\documentclass[11pt,twoside]{amsart}
\usepackage{amsfonts}
\usepackage{amsmath}
\usepackage{amssymb}
\usepackage{mathrsfs}

\newtheorem{theo}{Theorem}

\newtheorem{lem}{Lemma}
\newtheorem{prop}{Proposition}

\theoremstyle{definition}

\theoremstyle{remark}
\newtheorem{rem}{\bf Remark\/}

\newtheorem{exple}{\bf Example\/}
\numberwithin{equation}{section}

\def\1{{\mathchoice {\rm 1\mskip-4mu l} {\rm 1\mskip-4mu l}{\rm 1\mskip-4.5mu l} {\rm 1\mskip-5mu l}}}
\newcommand{\ds}{\displaystyle}
\newcommand{\w}{\wedge}

\title[On the directional Lelong-Demailly numbers]{On the directional Lelong-Demailly numbers of positive currents}

\author[M. Zaway]{Mohamed Zaway}
\email{mohamed\_zaway@yahoo.fr}
\author[H. Hawari]{Haithem Hawari}
\email{haithem.hawari@yahoo.fr}
\author[N. Ghiloufi]{Noureddine Ghiloufi}
\email{noureddine.ghiloufi@fsg.rnu.tn}
\address{Department of Mathematics\\ Faculty of sciences of Gab\`es \\ University of Gab\`es \\ Erriadh 6072 Gab\`es Tunisia.}
\begin{document}
\maketitle

\begin{abstract}
    In this paper we study the existence of  the directional Lelong-Demailly numbers of positive plurisubharmonic or plurisuperharmonic currents. We prove the independence of these numbers to the system of coordinates. Moreover these numbers will be given by locally integrable functions. \\

    \textbf{Sur les nombres de Lelong-Demailly  directionnels des courants positifs.}\\
    \textsc{R\'esum\'e.} Dans cet article on \'etudie l'existence des nombres de Lelong-Demailly directionnels des courants positifs plurisousharmoniques ou plurisurharmoniques. On montre l'ind\'ependance de ces nombres du syst\`eme des coordonn\'ees. De plus ces nombres seront donn\'es par des fonctions localement int\'egrables.
\end{abstract}

\section{Introduction}
    The problem of the existence of Lelong numbers of positive currents on open subsets of $\Bbb C^n$ is still open, it was considered firstly by Lelong in the 1950'th, who had solved it in the case of positive closed currents, then in 1982'th this problem was solved by Skoda \cite{Sk} in the case of positive plurisubharmonic (psh) currents. Replacing the standard K\"ahler form in the definition of the Lelong number by a $(1,1)-$form relative to a psh function, Demailly \cite{De} prove the existence of so called generalized Lelong numbers or Lelong-Demailly numbers for positive closed currents. Finally this problem was studied by the second author \cite{Gh} in case of positive plurisuperharmonic (prh) currents. In 1996, Alessandrini and Bassanelli \cite{Al-Ba} consider an analogous problem, the existence of the directional Lelong numbers of positive currents and they solved it in plurisubharmonic case. A generalized form of this problem appear then, (directional Lelong-Demailly numbers) and it is solved by Toujani \cite{To} in the same case. A natural question can be asked in this statement, is there a relationship between the two notions? Our goal is to study the case of positive prh currents where we show in particular that there is no relationship between the two notions.\\

    In the hole of this paper, we consider $n,\ m,\ k\in\Bbb N$ such that $k\leq n$. We use $(z,t)$ to indicate an element of $\Bbb C^N:=\Bbb C^n\times\Bbb C^m$. Let $\Omega:=\Omega_1\times\Omega_2$ be an open subset of $\Bbb C^N$ and $B$ be an open subset relatively compact  in $\Omega_2$. We consider two  $C^2$ positive  psh (psh) functions $\varphi$ and $v$ on $\Omega_1$ and $\Omega_2$ respectively such that $\log\varphi$ is psh in the open subset $\{\varphi>0\}$. We assume that the two functions $\varphi$ and $v$ are semi-exhaustive
    (i.e. there exists $R=R(\varphi)>0$ such that for every $c\in]-\infty,R[$, one has
    $\{z\in\Omega_1;\ \varphi(z)<c\}\subset\subset\Omega_1$).
    For $0<r<R$ and $r_1<r_2<R$, we set:
    $B_\varphi(r)=\{z\in\Omega;\ \varphi(z)<r\}$ and
    $B_\varphi(r_1,r_2)=\{z\in\Omega;\ r_1\leq\varphi(z)<r_2\}$.\\

    To simplify the notations, we set: $\beta_\varphi:=dd^c\varphi(z)$, $\beta_v:=dd^cv(t)$ and
    $\alpha_\varphi:=dd^c \log\varphi(z)$. In the classical case, we set $\omega_{z}=dd^c|z|^2$ and $\omega_{t}=dd^c|t|^2$ the K\"ahler forms  of $\Bbb C^n$ and $\Bbb C^m$ respectively, and  $\omega=\omega_{t}+\omega_{z}$ the corresponding one of $\Bbb C^N$.\\

    Let $T$ be a positive current of bidegree $(k,k)$ on $\Omega$. The purpose of this paper is to study the directional Lelong-Demailly number of $T$ with respect to $B$ relatively to  $\varphi$ and $v$ defined as the limit, $$\nu(T,\varphi,B ):=\lim_{r\to 0^+}\nu(T,\varphi,B,r)$$ where $\nu(T,\varphi,B,.)$ is the function defined by
        $$\nu(T,\varphi,B,r):=\frac{1}{r^{n-k}}\int_{B_\varphi(r)\times B}\ T\w\beta_\varphi^{n-k}\w\beta_v^m.$$

        In classical case, we use
        $$\nu(T,B):=\lim_{r\to0^+}\frac{1}{r^{2(n-k)}}\int_{\{|z|<r\}\times B}\ T\w\omega_z^{n-k}\w\omega_t^m.$$
   These numbers are studied in the case of positive psh current  by  Alessandrini and Bassanelli \cite{Al-Ba} and they showed that it's independent to the system of coordinates in the classical case. The general case has been studied  by Toujani \cite{To}. \\The main tool in this paper is the following formula:
    \begin{prop} (Lelong-Jensen Formula) (See \cite{To})
        We assume that the current $T$ is positive psh or prh on $\Omega$. Let $(T_j)_{j\in\Bbb N}$ be a
smooth regularizing sequence of  $T$. Then for every $0<r_1<r_2<R$  and for all $1\leq p\leq n-k$, $0\leq q <p$, one has

            $$\begin{array}{ll}
              &\ds\frac{1}{r_2^{q+1}}\ \int_{B_\varphi(r_2)\times B}\
              T_j\w{\alpha_\varphi}^{p-q-1}\w\beta_\varphi^{n-k-p+q+1}\w\beta_v^m \\
              &\hfill\ds -\frac{1}{ r_1^{q+1}}\ \int_{B_\varphi(r_1)\times B}\
              T_j\w\alpha_\varphi^{p-q-1}\w\beta_\varphi^{n-k-p+q+1}\w\beta_v^m\\
              =&\ds \int_{B_\varphi(r_1,r_2)\times B}\
              T_j\w{\alpha_\varphi}^p\w{{\beta_\varphi}^{n-k-p}}\w{\beta_v}^m\\
              &\ds +\int_{r_1}^{r_2}\  \left(\frac{1}{ s^{q+1}}-\frac{1}{ r_2^{q+1}}\right) ds\
              \int_{B_\varphi(s)\times B}\ dd^cT_j\w\alpha_\varphi^{p-q-1}\w\beta_\varphi^{n-k-p+q}\w\beta_v^m\\
              &\ds+\int_0^{r_1}\ \left(\frac{1}{ r_1^{q+1}}-\frac{1}{ r_2^{q+1}}\right)ds\ \int_{B_\varphi(s)\times B}\
              dd^cT_j\w\alpha_\varphi^{p-q-1}\w\beta_\varphi^{n-k-p+q}\w\beta_v^m\
            \end{array}$$
        This equality still  true when we replace $T_j$ by $T$ for $q=p-1$.
    \end{prop}
    As a consequence of this formula, if $T$ is a positive psh current on $\Omega$ then the function $\nu(T,\varphi,B,.)$ is positive and increasing on $]0,R[$ so its limit at 0, $\nu(T,\varphi,B)$, exists. This result is true when $m=0$ (Lelong-Demailly numbers).\\
    This statement is so different for positive prh currents as the following examples shown:\\
  \begin{exple}
     The current $T_0:=-\log|z_2|^2[z_1=0]$ is an example of positive prh current of bidimension $(1,1)$ on the unit ball of $\Bbb C^2$ which has no  Lelong number at 0, but it has (classical) directional Lelong numbers (in both directions) and this can answer to the question asked before.
  \end{exple}

\begin{exple}
  The current $T_1:=-\log(|z_1|^2+|z_3|^2)[z_2=0]$ is positive prh of bidimension $(2,2)$ on the unit ball $\Bbb B$ of $\Bbb C^3$ which has no directional Lelong number with respect to the disc $B:=\mathbb{D}(0,\frac12)\subset\{(0,0)\}\times\Bbb C$.
\end{exple}
        In fact, for $r<\frac{\sqrt{3}}2$ we have $\{|z_1|^2+|z_2|^2<r^2\}\times B\subset\subset \mathbb{B}$. Thanks to Fubini-Tonelli theorem we have
        $$\begin{array}{l}
            \nu(T_1,B,r)\\
            =\ds\frac{-1}{\pi^2r^2}\int_{\{|z_1|<r\}\times B}\log(|z_1|^2+|z_3|^2) \frac i2 dz_1\w d\overline{z}_1 \w \frac i2 dz_3\w d\overline{z}_3\\
            =\ds\frac 1{\pi^2r^2}\int_{B}\left(\int_{|z_1|<r}-\log(|z_1|^2+|z_3|^2) \frac i2 dz_1\w d\overline{z}_1\right)\frac i2 dz_3\w d\overline{z}_3\\
            =\ds\frac 1{\pi^2r^2}\int_{B}\left(\int_0^{2\pi}\int_0^r -t\log(t^2+|z_3|^2)dtd\theta \right)\frac i2 dz_3\w d\overline{z}_3\\
            =\ds\frac 2{\pi r^2}\int_{B}\left(-\frac12r^2 \log(r^2+|z_3|^2)+\int_0^r \frac{t^3}{t^2+|z_3|^2}dt\right)\frac i2 dz_3\w d\overline{z}_3\\
            =\ds\frac2{\pi r^2}\int_{B} \left(-\frac12r^2 \log(r^2+|z_3|^2)+ \frac{r^2}2-\frac{|z_3|^2}2\left[\log(r^2+|z_3|^2)-\log(|z_3|^2)\right]\right)\frac i2 dz_3\w d\overline{z}_3\\
            =\ds \frac 2{r^2}\int_0^{\frac12}\left(-\frac12r^2 \log(r^2+s^2)+\frac{r^2}2-\frac{s^2}2\left[\log(r^2+s^2)-\log(s^2)\right]\right)sds\\
            \geq\ds\frac{-\log(r^2+\frac14)}{128r^2}+\frac{r^2\log(r^2+\frac14)}2+\frac1{r^2}(-\frac{\log r}{64}-\frac 1{256}).
        \end{array}$$
         The result is obtained by tending  $r$ to 0.\hfill$\square$\\

    As the second example shows, we need a suitable condition to prove the existence of directional Lelong-Demailly numbers; for this, as it was introduced in \cite{Gh}, we say that $T$ ($B,\ \varphi$ and $v$) satisfy Condition $(C)$ if the function $s\longmapsto s^{-1}\nu(dd^cT,\varphi,B,s)$ is integrable in a neighborhood of 0. Using the Lelong-Jensen formula, we can   see that Condition $(C)$ is satisfied for a positive psh current.

    \begin{theo}\label{the1}
            If the current $T$ is  positive prh and satisfies Condition $(C)$ then the directional Lelong-Demailly number $\nu(T,\varphi,B)$ exists.
       \end{theo}
    \begin{proof}
        For $0<r<R$, we consider
        $$g(r):=\nu(T,\varphi,B,r)+\ds\int_0^r\left(\frac{s^{n-k}}{r^{n-k}}-1\right)\frac{\nu(dd^cT,\varphi,B,s)}{s}ds.$$

        Thanks to Condition $(C)$, the function $g$ is well defined. Moreover it is positive  on $]0,R[$ because the  function $\nu(dd^cT,B,.)$ is negative on $]0,R[.$\\
        As the current $T$ is  positive prh, then using  Lelong-Jensen formula, we have for $0<r_1<r_2<R$,
        $$\begin{array}{cl}
                &g(r_2)-g(r_1)\\
                =&\ds\nu(T,\varphi,B,r_2)-\nu(T,\varphi,B,r_1)+\ds \frac{1}{r_2^{n-k}}\int_0^{r_2}s^{n-k-1}\nu(dd^cT,\varphi,B,s)ds\\
                &\ds-\frac{1}{{r_1}^{n-k}}\int_0^{r_1}s^{n-k-1}\nu(dd^cT,\varphi,B,s)ds -\ds\int_{r_1}^{r_2}\frac{\nu(dd^cT,\varphi,B,s)}{s}ds\\
                =&\ds \int_{B_\varphi(r_1,r_2)\times B}T\w\alpha_\varphi^{n-k}\w \beta_v^m\geq0.
            \end{array}$$
        Hence the function $g$ is positive and increasing on $]0,R[$, thus its limit when $r$ goes to 0 exists. As Condition $(C)$ is satisfied and $(s/r)^{n-k}-1$  is uniformly bounded  in neighborhood of $0$, then
        $$\ds\lim_{r\rightarrow0^+}\int_0^r\left(\frac{s^{n-k}}{r^{n-k}}-1\right)\frac{\nu(dd^cT,\varphi,B,s)}{s}ds=0.$$
        Therefore
        $$\ds \lim_{r\rightarrow0^+} g(r)=\ds\lim_{r\rightarrow0^+}\nu(T,\varphi,B,r)=\nu(T,\varphi,B).$$
    \end{proof}
    \begin{exple}
        Let $z=(z_1,z_2)$ and $T_2(z,t):=d\log(|z_1|^2+|z_2|^2)\w d^c\log(|z_1|^2+|z_2|^2)$. The current $T_2$ is positive prh of bidegree $(1,1)$ on $\Bbb C^3$ and $dd^cT_2(z,t)=-[z=0]$. If $B:=\mathbb{D}(0,1)\subset \{0\}\times \Bbb C$ is the unit disc of $\Bbb C$ then
        $$\nu(dd^cT_2,B,r)=\int_{\{|z|<r\}\times B}-[z=0]\w dd^c|t|^2=-1$$
        and
        $$\begin{array}{lcl}
            \nu(T_2,B,r)&=&\ds\frac{1}{r^2}\int_{\{|z|<r\}\times B}T_2\w\omega_z\w\omega_t\\&=&\ds\frac{1}{8\pi^3r^2}\int_{\{|z|<r\}\times B}\frac1{|z|^2}idz_1\w d\overline{z}_1\w idz_2\w d\overline{z}_2\w idt\w d\overline{t}=1.
        \end{array}$$
    \end{exple}
    This example proves that Condition $(C)$ in Theorem \ref{the1} is not necessary for the existence of directional Lelong-Demailly numbers.

\section{Independence to the system of coordinates }
    In this part we prove the independence of Lelong-Demailly numbers of positive psh or prh currents to the system of coordinates, this result is due to  Alessandrini and Bassanelli \cite{Al-Ba} in the classical case.
    \begin{prop}\label{prop2}
        Let $T$ be a positive prh or psh current of bidegree $(k,k)$ on  $\Omega$ and $p\in[2,+\infty[$. For every $r\in]0,R(\varphi)]$, we have
            \begin{equation}\label{equ2.1}
                \begin{array}{l}
                \nu(T,\varphi^p,B,r^p)\\
                =\ds p^{n-k}\left(\nu(T,\varphi,B,r)+\int_0^r\frac{\nu(dd^cT,\varphi,B,s)}s \left(\frac{s^{n-k}}{r^{n-k}} -\frac{s^{(n-k)p}}{r^{(n-k)p}}\right)ds\right).
                \end{array}
            \end{equation}
        In particular, if the current $T$ is positive prh, then $\nu(T,\varphi,B)$ exists if and only if $\nu(T,\varphi^p,B)$ exists for some  $p\geq 2$.\\
        Furthermore, in both cases (with the assumption that $\nu (T,\varphi,B)$ exists if $T$ is prh) we have
        $$\nu(T,\varphi^p,B) = p^{n-k}\left( \nu (T,\varphi,B)+ \frac{p-1}{p(n-k)}\nu(dd^cT,\varphi,B)\right).$$
    \end{prop}
    \begin{proof}
        Let $\epsilon >0$. If we replace $\varphi$ by $\varphi_\epsilon=\varphi+\epsilon$ and $\psi$ by $\psi_\epsilon=(\varphi+\epsilon)^p$ in the Lelong-Jensen formula, we obtain for every  $0<r_1< \epsilon<r=r_2<R(\varphi)$,
        \begin{equation}\label{equ2.2}
         \begin{array}{l}
            \nu(T,\varphi_\epsilon,B,r) \\
            = \ds \int_{B_{\varphi_\epsilon}(\epsilon,r)\times B}T\w\alpha_{\varphi_\epsilon}^{n-k}\w\beta_v^m \ds-\frac 1{r^{n-k}}\int_\epsilon^r s^{n-k-1}\nu(dd^cT,\varphi_\epsilon,B,s)ds \\
            \hfill\ds+\int_\epsilon^r\frac{\nu(dd^cT,\varphi_\epsilon,B,s)}s ds
        \end{array}
        \end{equation}
        Using $dd^cT$ instead of $T$, Equality (\ref{equ2.2}) becomes
        \begin{equation}\label{equ2.3}
            \begin{array}{lcl}
              \nu(dd^cT,\psi_\epsilon,B,r^p) & =& \ds \int_{B_{\psi_\epsilon}(\epsilon^p,r^p)\times B} dd^cT\w\alpha_{\psi_\epsilon}^{n-k-1}\w\beta_v^m\\
              & = & p^{n-k-1}\ds \int_{B_{\varphi_\epsilon}(\epsilon,r)\times B}dd^cT\w\alpha_{\varphi_\epsilon}^{n-k-1}\w\beta_v^m\\
              &=&p^{n-k-1}\nu(dd^cT,\varphi_\epsilon,B,r)
            \end{array}
        \end{equation}
        Let now
            \begin{equation}\label{equ2.4}
                \begin{array}{l}
                    \nu(T,\psi_\epsilon,B,r^p)\\
                    = \ds \int_{B_{\psi_\epsilon}(\epsilon^p,r^p)\times B} T\w\alpha_{\psi_\epsilon}^{n-k}\w\beta_v^m + \int_{\epsilon^p}^{r^p}\left(1-\frac{s^{n-k}}{r^{p(n-k)}}\right) \frac{\nu(dd^cT,\psi_\epsilon,B,s)}s ds\\
                    = \ds \int_{B_{\psi_\epsilon}(\epsilon^p,r^p)\times B} T\w\alpha_{\psi_\epsilon}^{n-k}\w\beta_v^m + p\int_\epsilon^r\left(1-\frac{s^{p(n-k)}}{r^{p(n-k)}}\right) \frac{\nu(dd^cT,\psi_\epsilon,B,s^p)}s ds\\
                    = \ds p^{n-k}\int_{B_{\varphi_\epsilon}(\epsilon,r)\times B}T\w\alpha_{\varphi_\epsilon}^{n-k}\w\beta_v^m+ p^{n-k}\int_\epsilon^r\left(1-\frac{s^{p(n-k)}}{r^{p(n-k)}}\right) \frac{\nu(dd^cT,\varphi_\epsilon,B,s)}s ds\\
                    =\ds p^{n-k}\left\{\nu(T,\varphi_\epsilon,B,r)- \int_\epsilon^r\left(1-\frac{s^{n-k}}{r^{n-k}}\right) \frac{ \nu(dd^cT,\varphi_\epsilon,B,s)}sds\right\} \\
                    \ds \hfill + p^{n-k}\int_\epsilon^r\left(1-\frac{s^{p(n-k)}}{r^{p(n-k)}}\right) \frac{\nu(dd^cT,\varphi_\epsilon,B,s)}s ds\\
                    =\ds p^{n-k}\left(\nu(T,\varphi_\epsilon,B,r)+\int_\epsilon^r \frac{\nu(dd^cT,\varphi_\epsilon,B,s)}s \left(\frac{s^{n-k}}{r^{n-k}} -\frac{s^{p(n-k)}}{r^{p(n-k)}}\right)ds \right).
                \end{array}
            \end{equation}

      here we have used  successively Equality (\ref{equ2.2}) with $\psi_\epsilon$ instead of $\varphi_\epsilon$, then the change of variable $s\mapsto s^p$, next Equality (\ref{equ2.3}) and finally  equality (\ref{equ2.2}).\\
      When $\epsilon\to0$, Equality (\ref{equ2.4})  gives Equality (\ref{equ2.1}).\\

      Thanks to Equality (\ref{equ2.1}), for every $r \in]0,R(\varphi)[$, we have
      $$\ds\nu(T,\varphi,B,r)-\frac1{p^{n-k}}\nu(T,\varphi^p,B,r^p)
        =-\ds\int_0^r\nu(dd^cT,\varphi,B,s) \left(\frac{s^{n-k-1}}{r^{n-k}} - \frac{s^{p(n-k)-1}}{r^{p(n-k)}}\right)ds.$$
      For a technical  reason, we distinguish  the two cases:
      \begin{itemize}
            \item \textit{First case $dd^cT\geq0$:} The current $dd^cT$ is positive closed, so $\nu(dd^cT,\varphi,B,.)$ is  a positive increasing  function. Hence
                $$\begin{array}{lcl}
                    \ds\nu(dd^cT,\varphi,B)\frac{p-1}{p(n-k)}&\leq&\ds \int_0^r\nu(dd^cT,\varphi,B,s) \left(\frac{s^{n-k-1}}{r^{n-k}} - \frac{s^{p(n-k)-1}}{r^{p(n-k)}}\right)ds\\
                    &\leq& \ds \nu(dd^cT,\varphi,B,r)\frac{p-1}{p(n-k)}.
                \end{array}$$
                It follows that
                $$\begin{array}{lcl}
                    \ds- \frac{p-1}{p(n-k)}\nu(dd^cT,\varphi,B,r)&\leq&\ds\nu(T,\varphi,B,r)-\frac1{p^{n-k}} \nu(T,\varphi^p,B,r^p)\\
                    &\leq&\ds - \frac{p-1}{p(n-k)}\nu(dd^cT,\varphi,B).
                \end{array}$$
            \item \textit{Second case $dd^cT\leq0$:}
            The current $dd^cT$ is negative and closed, so $\nu(dd^cT,\varphi,B,.)$ is a negative  decreasing  function. A similar computation proves that we have
                $$\begin{array}{lcl}
                    \ds- \frac{p-1}{p(n-k)}\nu(dd^cT,\varphi,B) &\leq&\ds\nu(T,\varphi,B,r)-\frac1{p^{n-k}} \nu(T,\varphi^p,B,r^p)\\
                    &\leq&\ds - \frac{p-1}{p(n-k)}\nu(dd^cT,\varphi,B,r).
                \end{array}$$
        \end{itemize}
    In both cases, the two terms of the right hand and the left hand have the same limit when $r \to0^+$. This completes the proof of the proposition.
    \end{proof}
    \begin{rem}
        If the current $T$ is positive psh (resp. prh satisfying Condition $(C)$) then
        $$\nu(T,\varphi^p,B) = p^{n-k}\nu (T,\varphi,B).$$
        The current $T_2$ defined on  $\Bbb C^3$ by $$T_2(z,t):=d\log(|z_1|^2+|z_2|^2)\w d^c\log(|z_1|^2+|z_2|^2)$$ shows that this equality is not true if Condition $(C)$ is not satisfied; so the second equality in Proposition \ref{prop2} is sharp.
    \end{rem}

\begin{theo}\label{the2}
         Let $T$ be a positive psh or prh current of  bidegree $(k,k)$ on $\Omega$ and $\varphi, \  \psi$ two psh functions such that
         $$ \ds\liminf_{\varphi (z)\to 0} \frac {\log \psi (z)} {\log \varphi (z)} \geq \ell .$$
         Assume that $T,\ \psi$ and $B$ satisfy Condition $(C)$. Then $\nu(T,\psi,B) \geq \ell^{n-k}\nu(T,\varphi,B)$.\\
         In particular, if $\log \psi (z) \sim \ell \log \varphi(z)$ when $\varphi(z) \in \mathscr V(0)$ then $\nu (T,\psi,B) = \ell^{n-k }\nu (T,\varphi,B).$
     \end{theo}
     This Theorem  is known with "comparison theorem"; a such theorem was proved by Demailly in case of positive closed currents. This result allows us to prove the independence of directional Lelong-Demailly numbers with a  changement of the system of coordinates.

     \begin{proof}
        Replacing $\psi$ by $\psi^p$ and $\ell $ by $(p-1)\ell $ where $p\geq2$ large enough, we can assume that
        $$\liminf_{\varphi (z)\to 0} \frac {\log \psi (z)} {\log \varphi(z)}> \ell\geq 2.$$
        Hence we have $$\lim_ {\varphi(z) \to 0} \frac {\psi (z)} {\varphi(z)^\ell} = 0.$$
        So the function $\Psi_ \epsilon = \psi + \epsilon \varphi^\ell \underset {\varphi (z) \in \mathscr V(0)} \sim \epsilon \varphi^\ell $ for every $\epsilon> 0$. As $\ell \geq2$, $\Psi_ \epsilon$ is $\mathcal{C}^2$. Thanks to dominated convergence theorem,
        $$ \ds \lim_{\epsilon \to 0} \frac1{r^{n-k}} \int_ {\{\Psi_\epsilon<r\}\times B} T \w\beta_{\Psi_\epsilon}^{n-k}\w\beta_v^m =\frac{1}{r^{n-k}} \int_{\{\psi<r\}\times B} T \w\beta_{\psi}^{n-k}\w\beta_v^m$$ for every $r\in]0,R(\psi)[$.

        If the current $T$ is positive prh then the  function $g_\epsilon$ defined by
        $$g_\epsilon(r)=\nu(T,\Psi_\epsilon,B,r)+\ds\int_0^r\left(\frac{s^{n-k}}{r^{n-k}}-1\right) \frac{\nu(dd^cT,\Psi_\epsilon,B,s)}{s}ds\geq0$$ is increasing, thus
        $$ \begin{array}{lcl}
                g_\epsilon(r) &=& \ds\frac1{r^{n-k}} \int_ {\{\Psi_\epsilon <r\}\times B} T \w\beta_{\Psi_\epsilon}^{n-k}\w\beta_v^m+\ds\int_0^r\left(\frac{s^{n-k}}{r^{n-k}}-1\right) \frac{\nu(dd^cT,\Psi_\epsilon,B,s)}{s}ds  \\
                &\geq & \ds \lim_{\rho\to 0} \frac1 {\rho^{n-k}} \int_{\{\Psi_\epsilon <\rho\}\times B} T \w\beta_{\Psi_\epsilon}^{n-k}\w\beta_v^m \\
                &=&  \ds \lim_{\rho \to 0} \frac1{\rho^{n-k}} \int_ {\{\epsilon\varphi^\ell <\rho\}\times B} T \w\beta_{\Psi_\epsilon}^{n-k}\w\beta_v^m  \\
                & \geq & \ds \lim_{\rho \to 0} \frac1{\rho^{n-k}} \int_{\{\epsilon\varphi^\ell <\rho\}\times B} T \w\beta_{\epsilon\varphi^\ell}^{n-k}\w\beta_v^m = \nu (T, \varphi^\ell,B).
        \end{array}$$
        Where we use the fact that $\Psi_\epsilon \sim \epsilon \varphi^\ell$ and $T \w (dd^c(\psi + \epsilon\varphi^\ell))^{n-k} \geq T \w(\epsilon dd^c(\varphi^\ell))^ {n-k}.$\\
        The same  result can be obtained in the case of positive psh current by using $g_\epsilon(r)=\nu(T,\Psi_\epsilon,B,r).$\\

        If $\epsilon\to 0$, we obtain $\nu(T,\psi,B,r)\geq \nu(T,\varphi^\ell,B)$. Thus, if $r\to 0$,  $\nu(T,\psi,B) \geq \nu(T,\varphi^\ell,B)$.\\

        In particular if  $\log \psi(z)\sim\ell\log \varphi(z)$, for  $\epsilon>0$ (with $\ell(1-\epsilon)\geq2$), there exists $\eta>0$ such that if  $|\varphi(z)|<\eta$, then $\varphi(z)^{\ell(1+\epsilon)}\leq \psi(z)\leq \varphi(z)^{\ell(1-\epsilon)}$. So we can apply the previous inequality to obtain
        $$\ell^{n-k}(1+\epsilon)^{n-k}\nu(T,\varphi,B)\leq\nu(T,\psi,B)\leq \ell^{n-k}(1-\epsilon)^{n-k}\nu(T,\varphi,B).$$   When  $\epsilon\to0$, we obtain $\nu(T,\psi,B)=\ell^{n-k}\nu(T,\varphi,B).$
    \end{proof}

\section{Main result}

        The main result of this paper is the following theorem:
      \begin{theo}\label{the3}
            Let $T$ be a positive prh current of bidegree  $(k,k)$,  $0\leq k\leq n$ on $\Omega$. We assume that  $T$, $B_0$ satisfy Condition $(C)$. Then there exists an open subset $V\subset B_0$ and a function $f\in L^1_{loc}(V)$ such that
            $$\ds \nu(T,B)=\lim_{r\to0} \frac{1}{r^{2(n-k)}}\int_{\{|z|<r\}\times B}T\w \omega_z^{n-k}\w \omega_t^{m}=\int_Bf(t)\omega_t^m$$ for every open ball $B\subset\subset V$.
      \end{theo}
      The analogous of this result in the case of positive psh currents was proved by  Alessandrini and Bassanelli \cite{Al-Ba}. For $k=0$ or $k=n$ the proof is simple, indeed
      \begin{itemize}
        \item [$\bullet$] If $k=0$ then $h:=T$ is a positive prh function and $$ H(z):=\int_B h(z,t)\omega_t^m $$ is also prh, so $\nu(H,z)=H(z)$. Therefore
            $$\begin{array}{lcl}
                \ds \nu(h,B)&=&\ds\lim_{r\rightarrow 0}\frac{1}{r^{2n}}\int_{\{|z|<r\}\times B}h(z,t)\omega_z^n\w\omega_t^m= \lim_{r\rightarrow 0}\frac{1}{r^{2n}}\int_{|z|<r}H(z)\omega_z^n\\
                &=&\ds\int_B h(0,t)\omega_t^m.
              \end{array}$$
        \item [$\bullet$] If $k=n$, one has $$\nu(T,B)=\ds\lim_{r\rightarrow0} \int_{\{|z|<r\}\times B}T\w\omega_t^{m}.$$ Let $\Theta=\1_{\Omega \smallsetminus Y}T$, where $Y=\{0\}\times \Bbb C^m\cap\Omega$. The subset $Y$ is analytic in $\Omega$ of codimension $k=n$. Since $||T||_K$ is finite for every compact subset $K\subset\Omega$, the current $\Theta$ has a locally finite mass in neighborhood of every point of $Y$. As $T$ is a  positive prh current, then $T$ and $\Theta$ are $\Bbb C-$flat currents. So $T=\1_YT+\widetilde{\Theta}$, where $\widetilde{\Theta}$ is the trivial extension of $\Theta$ by 0 across $Y$. Thanks to the support theorem, there exists a positive prh function $f$ on $Y$ such that $\1_YT=f[Y]$. Hence $T=f[Y]+\widetilde{\Theta}$ and it follows that
            $$\ds\int_{\{|z|<r\}\times B}T\w\omega_t^m=\int_Bf(t)\omega_t^m+\int_{\{|z|<r\}\times B}\widetilde{\Theta}\w\omega_t^m$$
            for every open ball $B$ of $\Omega_2$ and $r>0$ small enough. But $$\int_{\{|z|<r\}\times B}\widetilde{\Theta}\w\omega_t^m\leq\int_{\{|z|<r\}\times B}\widetilde{\Theta}\w\omega^m=||\widetilde{\Theta}||(\{|z|<r\}\times B)$$ and
            $$\ds\lim_{r\to 0}||\widetilde{\Theta}||{(\{|z|<r\}\times B)}=||\widetilde{\Theta}||(B)=0$$ which gives the result.
    \end{itemize}
\subsection{Preliminary lemmas}
        To  prove the main result in case $1\leq k \leq n-1$,  we need a special transformation of coordinates (introduced by Siu \cite{Siu}) given by  $w=w(z,t)=(w_1,....,w_{m+n})$  such that $(z,w_I)=(z_1,....,z_n,w_{i_1},....,w_{i_m})$ form a system of  coordinates of $\Bbb C^N$ for every $1\leq i_1<i_2<.....<i_m\leq m+n$. We denote by $\omega_{I}:=dd^c|w_{I}|^2$.\\
        The first lemma is a version of Lelong-Jensen formula adapted to $\omega_I$:
        \begin{lem}\label{lem1}(See \cite{Al-Ba})
            Let $\psi$ be a smooth $(k,k)$-form on $\Omega$. For any increasing $m-$multiindex $I$ and every $ 1 \leq p \leq n-k,\ 0 \leq q < p,\  0 < r,c $, such that $\{(z,t)\in\Omega;\ |z|<r,\ |w_I|<c\}\subset\subset\Omega$, we have
            $$\begin{array}{ll}
              &\ds \int_{\{|z|<r,\ |w_I|<c\}}\  \psi\w{\theta_z}^p\w{{\omega_z}^{n-k-p}}\w\omega_I^m  \\
              =&\ds \frac{1}{r^{2(q+1)}}\int_{\{|z|<r,\ |w_I|<c\}}\     \psi\w{\theta_z}^{p-q-1}\w\omega_z^{n-k-p+q+1}\w\omega_I^m \\
              &\ds -\int_{0}^{r}\  \frac{1}{ s^{2(q+1)}} 2sds\ \int_{\{|z|<s,\ |w_I|<c\}}\ dd^c\psi\w{\theta_z}^{p-q-1}\w{\omega_z}^{n-k-p+q}\w{\omega_I}^m\\
              & \ds+\int_0^r\ \frac{1}{ r^{2(q+1)}}2sds\ \int_{\{|z|<s,\ |w_I|<c\}}\ dd^c\psi\w\theta_z^{p-q-1}\w\omega_z^{n-k-p+q}\w\omega_I^m
            \end{array}$$
where $\theta_z:=dd^c\log|z|^2$.
        \end{lem}
        The following  proposition is the crucial point in the proof of the main result:
        \begin{prop}\label{prop3}
            Let $T$ be a positive current on $\Omega$ satisfying Condition $(C)$ and $(T_j)_{j\in\Bbb N}$ its regularizing sequence. Then there exist $0 < r_0,c $ such that for any increasing $m-$multiindex $I$ and every $ 0 \leq p \leq n-k, \ 0<r\leq r_0$ we have  $\{(z,t)\in\Omega;\ |z|<r,\ |w_I|<c\}\subset\{\{|z|<2r\}\times B_0\}\subset\subset\Omega$ and
            $$ \ds \sup_j\int_{\{|z|<r,\ |w_I|<c\}}\  T_j\w\theta_z^p\w\omega_z^{n-k-p}\w\omega_I^m<+\infty.$$
            
        \end{prop}
        \begin{proof}
            For $p=0$, the result is clear. \\
            So let $ 1 \leq p < n-k$.
            Let $r_0>0$ such that $\{|z|<2r_0\}\times B_0\subset\subset\Omega$. There exists $c>0$, small enough, such that for any increasing $m-$multiindex $I$, $$\{(z,t)\in\Omega;\ |z|<r_0,\ |w_I|<c\}\subset\{|z|<2r_0\}\times B_0.$$ Then for every $0<r\leq r_0$ one has  $$\{(z,t)\in\Omega;\ |z|<r,\ |w_I|<c\}\subset\{|z|<2r\}\times B_0.$$
             If we choose  $q=p-1$ in previous lemma, we obtain
            $$\begin{array}{ll}
                &\ds\int_{\{|z|<r,\ |w_I|<c\}}\  T_j\w{\theta_z}^p\w{\omega_z^{n-k-p}}\w{\omega_I}^m \\
                \leq&\ds \frac{1}{r^{2p}}\int_{\{|z|<r,\ |w_I|<c\}}\ T_j\w\omega_z^{n-k}\w\omega_I^m\\
                &\ds -\int_{0}^{r}\  \frac{1}{ s^{2p}} 2sds\ \int_{\{|z|<s,\ |w_I|<c\}}\ dd^cT_j\w\omega_z^{n-k-1}\w\omega_I^m\\
                \leq&\ds \frac{1}{r^{2p}}\int_{\{|z|<r,\ |w_I|<c\}}\ T_j\w\omega_z^{n-k}\w\omega_I^m\\
                &\ds +\int_{0}^{r}\   2sds\ \int_{\{|z|<s,\ |w_I|<c\}}\ -dd^cT_j\w\theta_z^p\w\omega_z^{n-k-1-p}\w\omega_I^m\\
                \leq&\ds \frac{1}{r^{2p}}\int_{\{|z|<r,\ |w_I|<c\}}\ T_j\w\omega_z^{n-k}\w\omega_I^m\\
                &\ds +r^2\ \int_{\{|z|<r,\ |w_I|<c\}}\ -dd^cT_j\w\theta_z^p\w\omega_z^{n-k-1-p}\w\omega_I^m.
              \end{array}$$

              Furthermore, for $r$ outside a  set at most countable, we have
              $$\ds\lim_{j\to+\infty}\int_{\{|z|<r,\ |w_I|<c\}} T_j\w\omega_z^{n-k}\w\omega_I^m= \int_{\{|z|<r,\ |w_I|<c\}}T\w\omega_z^{n-k}\w\omega_I^m.$$
            Since $-dd^cT$ is a positive closed current, then  thanks to Alessandrini-Bassanelli \cite{Al-Ba}, we have
            $$\sup_j \int_{\{|z|<r,\ |w_I|<c\}}\ -dd^cT_j\w\theta_z^p\w\omega_z^{n-k-1-p}\w\omega_I^m<+\infty.$$
            In the case $p = n-k$, it is easy to see that there exist $A>0$ and a closed $\mathcal C^\infty$ form $\phi$ on $\{|z|<2r_0\}\times B_0$ such that $\omega_I^m\leq A(\omega_z\w\phi+\omega_t^m)$ in neighborhood of $\{|z|<r_0,\ |w_I|<c\}$. So

            $$\begin{array}{l}
                \ds \sup_j \int_{0}^{r}\  \frac{1}{ s^{2(n-k)}} 2sds\ \int_{\{|z|<s,\ |w_I|<c\}}\ -dd^cT_j \w\omega_z^{n-k-1}\w\omega_I^m\\
               \leq\ds \sup_j \int_{0}^{r}\  \frac{A}{ s^{2(n-k)}} 2sds\ \int_{\{|z|<s,\ |w_I|<c\}}\ -dd^cT_j \w\omega_z^{n-k-1}\w(\omega_z\w\phi+\omega_t^m)\\
               \leq\ds \sup_j \int_{0}^{r}\  \frac{A}{ s^{2(n-k)}} 2sds\ \int_{\{|z|<2s\}\times B_0}\ -dd^cT_j \w\omega_z^{n-k-1}\w(\omega_z\w\phi+\omega_t^m)\\
               \leq\ds 2^{n-k}A\left(\int_{0}^{r} 2sds\int_{\{|z|<2s\}\times B_0}\ -dd^cT \w\theta_z^{n-k}\w\phi+\int_{0}^{r}\  \frac{\nu(-dd^cT,B,2s)}s ds\right).\\
            \end{array}$$
            Thanks to Condition $(C)$,  the last integral is finite.
      \end{proof}

    \begin{rem}(See \cite{Al-Ba})
        Let $U:=\cap_I\{(z,t)\in\Omega;\ |z|<r_0,\ |w_I|<c\}$, then thanks to the last proposition, there exist a subsequence of $(T_j)_{j \in \Bbb N}$, noted in the same way, and some currents $T^{(0)},T^{(1)},...,T^{(n-k)}$ defined on $U$ such that
        $$\ds \lim_{j\to+\infty}\left(\widetilde{T_j\w \theta_z^p}\right)=T^{(p)}$$
        weakly on $U$.
     \end{rem}
      \begin{lem}\label{lem2}
           With the same notations as in Proposition \ref{prop3}, if we denote by $$(dd^cT)^{(p-q-1)}:=\lim_{j\to+\infty}\widetilde{dd^cT_j\w\theta_z^{p-q-1}}$$ then we have
           \begin{equation}\label{equ3.1}
           \begin{array}{l}
               \ds\int_0^r \frac{sds}{s^{2(q+1)}}\int_{\{|z|<s,\ |w_I|<c\}}- (dd^cT)^{(p-q-1)}\w\omega_z^{n-k-p+q}\w\omega_I^m \\
               \ds\leq \lim_{j\to+\infty}\int_0^r \frac{sds}{s^{2(q+1)}}\int_{\{|z|<s,\ |w_I|<c\}}- dd^cT_j\w\theta_z^{p-q-1}\w\omega_z^{n-k-p+q}\w\omega_I^m<+\infty.
             \end{array}
           \end{equation}

       \end{lem}
       \begin{lem}\label{lem3}
             Let $R_j^p:=\log|z|dd^cT_j\w\theta_z^p$. There exist $S^{(0)},.....,S^{(n-k-1)}$ positive prh currents on  $U$ such that
            $$\ds\lim_{j\to+\infty}\widetilde{R_j^p}= S^{(p)}.$$
            Furthermore,  for every open ball $B$ and $r>0$ such that $\{|z|<r\}\times B\subset U$ one has
            $$\begin{array}{l}
                \ds \lim_{j\to+\infty}\int_0^r \frac{sds}{s^{2}}\int_{\{|z|<s\}\times B}-dd^c T_j\w\theta_z^{p-1}\w\omega_z^{n-k-p}\w\omega_t^m \\
                \ds=-\log r\int_{\{|z|<r\}\times B}(dd^c T)^{(p-1)}\w\omega_z^{n-k-p}\w\omega_t^m + \int_{\{|z|<r\}\times B} S^{(p-1)}\w\omega_z^{n-k-p}\w\omega_t^m.
              \end{array}$$
      \end{lem}
      \begin{proof}
        The proof of this lemma is similar to lemma 3.3 in \cite{Al-Ba}.
      \end{proof}
      \subsection{Proof of the main result}
      \begin{lem}\label{lem4}
            If $T^{(1)}$ and $S^{(0)}$ are as in previous lemma, then
            $$\nu(T,B)=\nu(T^{(1)},B)-2\nu(S^{(0)},B).$$
      \end{lem}
      \begin{proof}
            Thanks to Lelong-Jensen formula, we have
             \begin{eqnarray}
              &&\ds \int_{\{|z|<r\}\times B}\  T_{j}\w{\theta_z}\w\omega_z^{n-k-1}\w\omega_t^m     \label{3.2}\\ &= & \ds \frac {1}{r^{2}}\int_{\{|z|<r\}\times B}\  T_{j} \w\omega_z^{n-k}\w\omega_t^m \label{3.3} \\
                & &+\ds \int_{0}^{r}\  \frac{1}{s^2}2sds \int_{\{|z|<s\}\times B}\ - dd^cT_{j}\w\omega_z^{n-k-1}\w{\omega_t}^m \label{3.4} \\
               &  &-\ds \int_0^r \frac{1}{ r^{2}}2sds\ \int_{\{|z|<s\}\times B}\ - dd^cT_{j} \w \omega_z^{n-k-1}\w\omega_t^m\label{3.5}
            \end{eqnarray}
            It is easy to see that  $$\lim_{j\rightarrow+\infty} (\ref{3.2}) =\ds \int_{\{|z|<r\}\times B}\  T^{(1)}\w\omega_z^{n-k-1}\w\omega_t^m,$$
            $$\lim_{j\rightarrow+\infty} (\ref{3.3}) =\ds\frac{1}{r^2} \int_{\{|z|<r\}\times B}\  T\w\omega_z^{n-k}\w\omega_t^m$$
            for $r$ outside a subset which is at most countable.\\

            Set $G_j$ the function defined by
            $$G_j(s):= \int_{\{|z|<s\}\times B}- dd^cT_j\w\omega_z^{n-k-1}\w\omega_t^m.$$
            The function $G_j$ is positive and increasing, so thanks to dominated convergence  theorem, we get  $$\lim_{j\to+\infty} (\ref{3.5}) = \ds \int_0^r \frac{1}{ r^{2}}2sds \int_{\{|z|<s\}\times B}- dd^cT \w \omega_z^{n-k-1}\w\omega_t^m.$$
            Thanks to Lemma \ref{lem3}, one has
            $$\ds\lim_{j\rightarrow+\infty} (\ref{3.4})=\ds -2\log r\int_{\{|z|<r\}\times B}dd^c T\w\omega_z^{n-k-1}\w\omega_t^m + \int_{\{|z|<r\}\times B}2 S^{(0)}\w\omega_z^{n-k-1}\w\omega_t^m.$$
            Hence
            \begin{eqnarray}
               \mathscr A&:=&\ds \frac{1}{r^{2(n-k-1)}} \int_{\{|z|<r\}\times B}\ \left( T^{(1)}- 2S^{(0)}\right)\w\omega_z^{n-k-1}\w\omega_t^m\label{3.6}\\
               &=& \ds \frac{1}{r^{2(n-k)}} \int_{\{|z|<r\}\times B}\  T\w\omega_z^{n-k}\w\omega_t^m \label{3.7} \\
               && - \frac{2\log r}{r^{2(n-k-1)}} \int_{\{|z|<r\}\times B}\  dd^cT\w\omega_z^{n-k-1}\w\omega_t^m \label{3.8}\\
               && \ds - \frac{1}{r^{2(n-k)}}\int_0^r 2 s ds \int_{\{|z|<s\}\times B}\  - dd^cT\w\omega_z^{n-k-1}\w\omega_t^m.\label{3.9}
            \end{eqnarray}

            If we set $$G(s):=\ds\int_{\{|z|<s\}\times B}\  - dd^cT\w{{\omega_z}^{n-k-1}}\w\omega_t^m,$$ then by  Lemma \ref{lem2} with $p=n-k$ and $q=p-1$, the function  $s\longmapsto s^{1-2(n-k)}G(s)$ is integrable on $[0,r]$; it follows that
            $$ \ds\frac{1}{r^{2(n-k)}}\int_0^r sG(s)ds\leq \int_0^r\frac{sG(s)}{s^{2(n-k)}}ds<+\infty.$$
            As this last integral tends to 0 when $r\to 0$, we obtain  $\lim_{r\to 0} (\ref{3.9}) =0$.\\
            We remark that  $T^{(1)}$ and $S^{(0)}$ are positive prh currents, and by Lemma \ref{lem1},
            $$\begin{array}{lcl}
                \ds \int_0^r\frac{\nu((dd^cT)^{(1)},B,s)}sds&
                =&\ds\int_0^r\frac{ds}{s^{2(n-k-2)+1}} \lim_{j\to+\infty}\int_{\{|z|<s\}\times B}dd^cT_j\w \theta_z\w \omega_z^{n-k-2}\w\omega_t^m\\
                &=&\ds\int_0^r\frac1{s^{2(n-k-1)+1}}\lim_{j\to+\infty}\int_{\{|z|<s\}\times B}dd^cT_j\w \omega_z^{n-k-1}\w\omega_t^m\\
                &=&\ds\int_0^r\frac{\nu(dd^cT,B,s)}sds.
           \end{array}$$
           With a similar computation for $S^{(0)}$, we conclude that $T^{(1)}$and $S^{(0)}$ satisfy  Condition $(C)$. So by Theorem  \ref{the1}, $\nu(T^{(1)},B)$ and $\nu(S^{(0)},B)$ exist (i.e. $\lim_{r\to0}(\ref{3.6})$ exists). With the same reason, $\lim_{r\to 0} (\ref{3.7}) =\nu(T,B)$ is finite, so (\ref{3.8}) has a finite limit when $r\to 0$. Let
           $$\ds \mathfrak a:=\lim_{ r\to0}(\ref{3.8})=\lim_{ r\to0}\frac{2\log r}{r^{2(n-k-1)}} \int_{\{|z|<r\}\times B}\ - dd^cT\w{{\omega_z}^{n-k-1}}\w\omega_t^m.$$
           Assume that $\mathfrak a\not=0$,
           $$\ds \mathfrak a=\lim_{r\to0}r\log r\frac{2r G(r)}{r^{2(n-k)}}\quad \Longrightarrow\quad \ds\frac{\mathfrak a}{r\log r}\ \underset{\vartheta(0)}\cong\frac{2r G(r)}{r^{2(n-k)}}.$$
           So  $\int_0^{\frac{1}{2}}\frac{\mathfrak a}{s\log s }ds$ converges if and only if  $\int_0^{\frac{1}{2}}\frac{2s G(s)}{s^{2(n-k)}}ds$ converges, which is in contradiction because $$\ds\int_0^{\frac{1}{2}}\frac{\mathfrak a}{s\log s }ds=\mathfrak a[\log|\log s|\ ]_0^{\frac12}=+ \infty\quad \hbox{ and }\quad \ds\int_0^{\frac{1}{2}}\frac{2sG(s)}{s^{2(n-k)}}ds<+\infty;$$ hence $\mathfrak a=0$ and we obtain
            $$\nu(T^{(1)},B)-2 \nu(S^{(0)},B)=\nu(T,B).$$
      \end{proof}
      Now we can finish the proof of the main result. We have
      $$\begin{array}{l}
            \nu(T,B)\\
            =\ds\lim_{r\to0}\lim_{j\to+\infty}\frac{1}{r^{2(n-k-1)}}\int_{\{|z|<r\}\times B}\left(T_j\w\theta_z-2\log|z|dd^cT_j\right)\w\omega_z^{n-k-1}\w\omega_t^m\\
            =\ds\lim_{r\to0}\lim_{j\to+\infty}\int_{\{|z|<r\}\times B}\left(T_j\w\theta_z-2\log|z|dd^cT_j\right)\w\theta_z^{n-k-1}\w\omega_t^m\\
            =\ds\lim_{r\to 0}\int_{\{|z|<r\}\times B}T^{(n-k)}\w\omega_t^m-2\int_{\{|z|<r\}\times B}S^{(n-k-1)}\w\omega_t^m.
            \end{array}$$
            By the proof of the main theorem in the case $k=n$, we can deduce that there exist two positive functions $f_1,\ f_2\in L^1_{loc}(V)$ with $V:=U\cap \{0\}\times \Bbb C^m$ such that  $$\nu(T,B)=\ds\int_B(f_1(t)-2f_2(t))\omega_t^m.$$
            So
            $$\ds\nu(T,B)=\int_Bf(t)\omega_t^m $$
             where $f\in L^1_{loc}(V).$

\section*{acknowledgement}
    The authors would like to thank Professors Jean-Pierre Demailly and Khalifa Dabbek for many fruitful discussions     concerning this paper.


\begin{thebibliography}
{X-XX1}
\bibitem{Al-Ba}\textbf{L. Alessandrini, G. Bassanelli}, Lelong numbers of positive plurisubharmonic currents, Results Math. \textbf{30} (1996).
\bibitem{De}\textbf{J.-P. Demailly}, Nombres de Lelong g\'en\'eralis\'es, th\'eor\`eme d'int\'egrabilit\'e et d'analyticit\'e, Acta Math. \textbf{159}   (1987) 153-169.
\bibitem{Gh}\textbf{N. Ghiloufi}, On the Lelong-Demailly numbers of plurisubharmonic  currents, C. R. Acad. Sci. Paris, Ser. I,  (2011) 505-510.
\bibitem{Siu}\textbf{Y. T. Siu}, Analyticity of sets associated to Lelong numbers and the extension of closed positive currents, Invent. Math. \textbf{27},  (1974) 53-156.
\bibitem{Sk}\textbf{H. Skoda}, Prolongement des courants positifs ferm\'es de masse finie, Inv. Math.\textbf{66},  (1982) 361-376.
\bibitem{To}\textbf{M. Toujani}, Nombre de Lelong directionnel d'un courant positif plurisousharmonique, C. R. Acad. Sci. Paris, Ser. I 343  (2006) 705-710.
\end{thebibliography}
\end{document}